\documentclass[a4paper,reqno,12pt]{amsart}
\usepackage{amsmath,amsfonts,amssymb,pdfsync}
\usepackage[latin1]{inputenc}
\usepackage{color}
\usepackage{graphicx} % inserir imagens
\newtheorem{theorem}{Theorem}[section]

\newtheorem{remark}[theorem]{Remark}

\newcommand{\R}
 {\mathbb{R}}

\numberwithin{equation}{section}
\begin{document}

\title[Error estimate for a Neumann problem]{Error estimates for a Neumann problem in highly oscillating thin domains}

\author[M. C. Pereira]{Marcone C. Pereira$^\dag$}
\thanks{$^\dag$Partially
supported by CNPq 305210/2008-4 and FAPESP 2008/53094-4 and 2010/18790-0, Brazil}
\address[Marcone C. Pereira]{Escola de Artes, Ci\^encias e Humanidades,
Universidade de S\~ao Paulo, S\~ao Paulo SP, Brazil}
\email{marcone@usp.br}

\author[R. P. Silva]{Ricardo P. Silva$^\star$}
\thanks{$^\star$Partially
supported by FAPESP 2008/53094-4, PROPe/UNESP, Brazil}
\address[Ricardo P. Silva]{Instituto de Geoci\^{e}ncias e Ci\^{e}ncias Exatas, Universidade Estadual Paulista, Rio Claro SP, Brazil}
\email{rpsilva@rc.unesp.br}

\date{}


\begin{abstract}
 
In this work we analyze the convergence of solutions of the Poisson equation with Neumann boundary conditions in a two-dimensional thin domain with highly oscillatory behavior. We consider the case where the height of the domain, amplitude and period of the oscillations are all of the same order, and given by a small parameter $\epsilon>0$. Using an appropriate \emph{corrector} approach, we show strong convergence and give error estimates when we replace the original solutions by the first-order expansion through the \emph{Multiple-Scale Method}. 

\end{abstract}

\maketitle

%%%------------------------------------

\section{Introduction}

Several important problems arising in physics and engineering lead to consider boundary-value problems in domains with oscillating boundaries, such as flows over rough walls, electromagnetic waves in a region containing a rough interface \cite{N-K:97}, or heat transmission in winglets \cite{ABMG}.

In the present paper we consider the Poisson equation with Neumann boundary conditions, arising from the study of stead state of reaction-diffusion flow over a thin bar with a rough boundary \cite{ACPS}. The aim of our study is to derive precise estimates for the effective behavior of the solutions of such problem.

In order to setup the problem, let $g \in C^1(\R, \R)$ be a $L$-periodic positive function. Given a small parameter $\epsilon > 0$, let us consider the family of two-dimensional domains
\begin{equation} \label{TDG}
R^\epsilon = \{ (x_1,x_2) \in \R^2 \; | \; 0 < x_1 < 1,\;  0 < x_2 < \epsilon \, g({x_1}/{\epsilon}) \}.
\end{equation}
We restrict our attention to the solutions $w^\epsilon$ of the family of elliptic equations 
\begin{equation} \label{P}
\left\{
\begin{gathered}
- \Delta w^\epsilon + w^\epsilon = f^\epsilon
\quad \textrm{ in } R^\epsilon \\
\frac{\partial w^\epsilon}{\partial N^\epsilon} = 0
\quad \textrm{ on } \partial R^\epsilon,
\end{gathered}
\right.
\end{equation} 
where $N^\epsilon$ denotes the unit outward normal vector field to $\partial R^\epsilon$ and $f^\epsilon$ is a non-homogeneous term in $L^2 (R^\epsilon)$ uniformly bounded.

\begin{figure}[htp]
\centering \scalebox{0.7}{\includegraphics{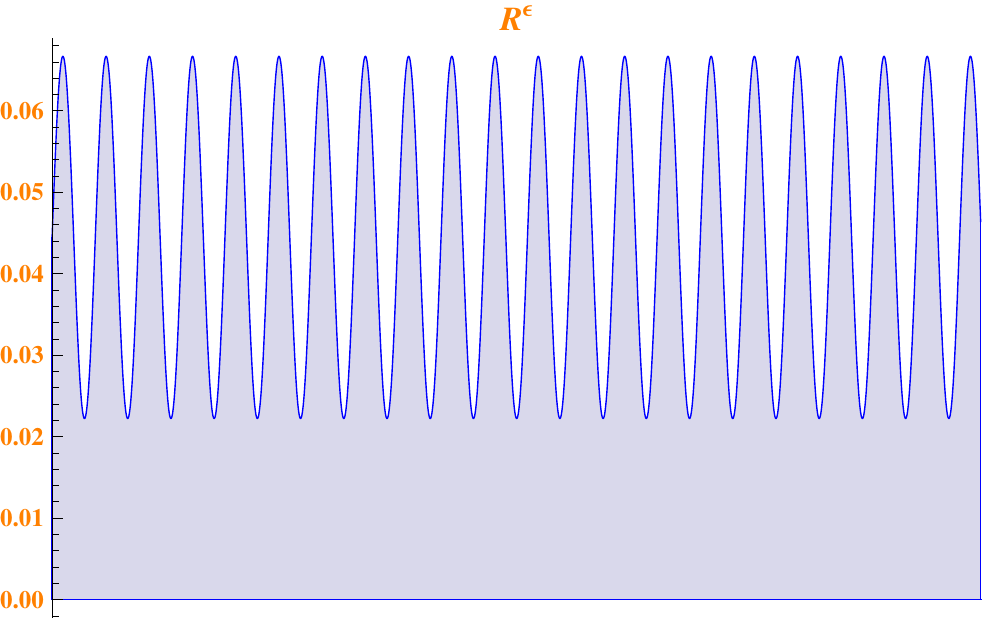}}
\label{thin-domain}
\caption{The thin domain $R^\epsilon$.}
\end{figure}

Many authors have devoted efforts in the investigation of the asymptotic behavior of a similar  equation in related domains, e.g. \cite{ABMG, BCh, DP, TAM, N-K:97}. The basic assumption is that the scale of wavelength of the roughness is small compared with the roughness amplitude. Therefore it is important to emphasize that the amplitude and period of the oscillations of $R^\epsilon$ are of the same order $\epsilon$, which also coincides with the order of thickness of the thin domain. This scaling makes the problem very resonant and the determination of the limiting problem ($\epsilon=0$) is not straightforward. It makes our analysis and results different from those papers.

Going back to problem \eqref{P}, in \cite{ACPS} the authors combine methods from \emph{homogenization theo\-ry}, specially those related to re\-ti\-cu\-lated structures, to obtain the limiting behavior of the family of solutions $w^\epsilon$.   We recall that the convergence obtained in \cite{ACPS} is with respect to the weak topology of the Sobolev space $H^1$. As pointed in \cite{A}, this convergence is the best one in these spaces. In general, one can not expect strong convergence in homogenization related problems. We refer the reader to \cite{BLP,CP-79,SP,Tt} for a general introduction to the theory of homogenization and to \cite{CP} for a general treatise on reticulated structures. 

In such problems, where the analytical solution demands a substantial effort or is unknown, it is of interest in discussing error estimates (in norm) when we replace it with numerical approximations.

In this work we made use of an appropriate \emph{corrector} approach developed by Bensoussan, Lions and Papanicolaou in \cite{BLP} to derive a kind of strong convergence and error estimates in $H^1$-norm. This can be summarized as: 

\emph{Suppose $w^\epsilon \rightharpoonup w_0$ weakly in $H^1(R^\epsilon)$ in a sense to be defined. It is possible to find an explicit corrector term $\kappa^\epsilon \in H^1(R^\epsilon)$, $\kappa^\epsilon = o(\epsilon)$ in $L^2(R^\epsilon)$, in order to}
\begin{equation*} \label{SC}
\epsilon^{-1/2} \|w^\epsilon - w_0 - \kappa^\epsilon \|_{H^1(R^\epsilon)} \leq C \,  \epsilon^{1/2}
\end{equation*}
for $\epsilon>0$ small enough.
The reader is invited to \cite{CD} for a classical introduction to correctors approach,  \cite{Madureira1, Madureira3, GP} for further discussion on approximations of arbitrary order for related problems, and \cite{CDG,Tt} for recent works in this subject as well. 

Still related with the limiting behavior of problems in thin domains but in connection with continuity properties of global attractors associated to parabolic problems, we can cite  \cite{Elsken:2005p40, HR, PR01, Prizzi:2002p441,R, Ricardo} and their references.

Finally, we would like to observe that although we just treat the Neumann problem for the Poisson equation,  we may also consider different conditions in the lateral boundaries of the thin
domain $R^\epsilon$, while preserving the Neumann type boundary condition in the upper and lower boundary. The limiting problem preserves the boundary condition. Note that Dirichlet boundary condition in the upper and lower boundaries is a trivial case since the trace operator is continuous on the set $(0,1) \subset \partial R^\epsilon$. 

This paper is organized as follows. In Section \ref{PRE}, we introduce the functional setting for the perturbed problem \eqref{P} and the limiting one. In Section \ref{MSM}, we provide formal computations to derive the homogenized problem by the Multiple-Scale Method. In Section \ref{sec:fir-corr}, we obtain convergence in $H^1$-norms for the first-order corrector, and in Section \ref{EE}, we use the second-order corrector to obtain rates of convergence for the first order approximation. We emphasize the relevance of these rates for numerical approximations of solutions of partial differential equations in highly heterogeneous and porous media.

%%%------------------------------------

\section{Preliminaries} \label{PRE}

We stress for the fact that $R^\epsilon$ varies in accordance with a positive parameter $\epsilon$ and, when $\epsilon$ goes to $0$, the domains $R^\epsilon$ collapse themselves to the  interval $(0,1)$. Therefore, in order to preserve the ``relative capacity'' of a mensurable subset $\mathcal{O} \subset R^{\epsilon}$, we rescale the Lebesgue measure by ${1}/{\epsilon}$, dealing with the singular measure 
$$
\rho_{\epsilon}(\mathcal{O})=\epsilon^{-1} | \mathcal{O} |.
$$ 
This measure has been considered in studies involving thin domains, e.g. \cite{HR,PR01,R}, and allows us introduce the Lebesgue $L^2(R^\epsilon; \rho_{\epsilon})$ and the Sobolev $H^1(R^\epsilon;\rho_{\epsilon})$ spaces. The norms in these spaces will be denoted by $||| \cdot |||_{L^2(R^\epsilon)}$ and $||| \cdot |||_{H^1(R^\epsilon)}$ respectively, being induced by the inner products
\begin{equation*}
(u,v)_\epsilon = \epsilon^{-1} \int_{R^\epsilon} u \, v \, dx, \quad \forall \, u,v \in L^2(R^\epsilon)
\end{equation*} 
and
\begin{equation*}
a_\epsilon(u,v) = \epsilon^{-1} \int_{R^\epsilon} \{ \nabla u \cdot \nabla v + u \, v \} dx, \quad \forall \, u,v \in H^1(R^\epsilon)
\end{equation*} 
respectively.

\begin{remark}
Notice that the $||| \cdot |||$ - norms and the usual ones in $L^2(R^\epsilon)$ and $H^1(R^\epsilon)$ are equivalents and easily related by
$$
\begin{gathered}
||| u |||_{L^2(R^\epsilon)} = \epsilon^{-1/2} \| u \|_{L^2(R^\epsilon)}, \quad \forall \, u \in L^2(R^\epsilon) \\
||| u |||_{H^1(R^\epsilon)} = \epsilon^{-1/2} \| u \|_{H^1(R^\epsilon)}, \quad \forall \,u \in H^1(R^\epsilon).
\end{gathered}
$$
\end{remark}

The variational formulation of (\ref{P}) is: find $w^\epsilon \in H^1(R^\epsilon)$ such that 
\begin{equation} \label{VFP}
\int_{R^\epsilon} \Big\{ \nabla w^\epsilon \cdot \nabla \varphi 
+ w^\epsilon \, \varphi \Big\} dx_1 dx_2 = \int_{R^\epsilon} f^\epsilon \, \varphi \, dx_1 dx_2, \quad \forall \, \varphi \in H^1(R^\epsilon),
\end{equation}
which is equivalent to find $w^\epsilon \in H^1(R^\epsilon;\rho^\epsilon)$ such that 
\begin{equation}\label{eq:vfor}
a_\epsilon(\varphi, w^\epsilon) = (\varphi, f^\epsilon)_\epsilon,  \quad \forall \, \varphi \in H^1(R^\epsilon;\rho_\epsilon). 
\end{equation}
Observe that the solutions $w^\epsilon$ satisfy a priori estimates uniformly  with respect to $\epsilon$. Thus, we can take $\varphi = w^\epsilon$ in (\ref{VFP}) and (\ref{eq:vfor}) to obtain
\begin{equation} \label{priori}
\begin{gathered}
\| \nabla w^\epsilon \|_{L^2(R^\epsilon)}^2
+ \| w^\epsilon \|_{L^2(R^\epsilon)}^2
\le \| f^\epsilon \|_{L^2(R^\epsilon)} \| w^\epsilon \|_{L^2(R^\epsilon)} \\
||| \nabla w^\epsilon |||_{L^2(R^\epsilon)}^2
+ ||| w^\epsilon |||_{L^2(R^\epsilon)}^2
\le ||| f^\epsilon |||_{L^2(R^\epsilon)} ||| w^\epsilon |||_{L^2(R^\epsilon)}.
\end{gathered}
\end{equation}

In order to capture the limiting behavior of $a_\epsilon(w^\epsilon, w^\epsilon)$ as $\epsilon \to 0$, we consider the sesquilinear form $a_0$ in $H^1(0,1)$ given by
\begin{equation} \label{IPA0}
a_0(u,v)  = \hat{g} \int_0^1  \left\{ \dfrac{d u}{dx}  \dfrac{d v}{dx} +  u \, v \right\} dx, \quad \forall \, u,v \in H^1(0,1),
\end{equation}
where
$\hat{g}= \frac{1}{L} \int_0^L g(s) \, ds$ is the average of function $g$ on the interval $(0,L)$. 
We also will consider $L^2(0,1)$ endowed with the norm induced by the  inner product $(\cdot, \cdot)_0$, given by 
\begin{equation} \label{IP0}
(u,v)_0 = \hat{g} \int_{0}^1 u \, v \, dx, \quad \forall \, u,v \in L^2(0,1).
\end{equation}
%

%--------------------------------------------------------

\section{Multiscale Expansion of Solutions} \label{MSM}

In this section we proceed as in \cite{ACPS,BLP,CD,CP,SP}. 
We suppose for a moment $f^\epsilon(x_1,x_2) = f(x_1)$ for all $\epsilon > 0$ and $(x_1,x_2) \in R^\epsilon$. Then we use the Multiple-Scale Method to introduce formally the \emph{homogenized equation} and the \emph{second-order expansion} to the solutions $w^\epsilon \in H^1(R^\epsilon)$ of problem \eqref{P}. We seek for a formal asymptotic expansion of the form
\begin{equation} \label{FE}
w^\epsilon(x_1,x_2) = w_0 \left(x_1,\frac{x_1}{\epsilon},\frac{x_2}{\epsilon} \right)
+\epsilon w_1 \left(x_1,\frac{x_1}{\epsilon},\frac{x_2}{\epsilon} \right) 
+ \epsilon^2 w_2 \left(x_1,\frac{x_1}{\epsilon},\frac{x_2}{\epsilon}\right) + \ldots
\end{equation}
The variables $x_1$ and $x_2$ represent the ``macroscopic" scale on the model, while $x_1/\epsilon$ and $x_2/\epsilon$ represent the ``microscopic" effect of the oscillating phenomena on the thin domain.

The fact that $R^\epsilon$ degenerates to a line segment when $\epsilon$ goes to $0$ suggests that the solutions $w^\epsilon$ will not depend on the ``macroscopic" variable $x_2$. This is taken into account in the choice of $w_i$ assuming that $w_i$ does not depend on the macroscopic variable $x_2$. 
 
In order to construct the functions $w_i$ of \eqref{FE}, motivated by the periodic nature of $R^\epsilon$, we consider the basic cell 
\begin{equation} \label{Y*}
Y^*=\{ (y,z) \in \R^2 \; | \; 0<y<L, \; 0<z<g(y)\}.
\end{equation} 
We decompose $\partial Y^*$ in  the lateral part of the boundary $B_0 =\{(0,z) \in \R^2 \, | \, 0<z<g(0)\} \cup \{(L,z)) \in \R^2 \, | \, 0<z<g(L)\}$, the upper part of the boundary $B_1=\{(y, g(y)) \in \R^2 \, | \, 0<y<L\}$ and the lower part of the boundary $B_2 = \{(y,0)) \in R^2 \, | \, 0<y<L\}$. We also assume that $w_i=w_i(x,y,z)$ is a $L$-periodic function in the variable $y$ (ie. $w_i(x,y+L,z) = w_i(x,y,z)$) which is defined for all $x \in (0, 1)$ and $(y,z) \in Y^*$.

Performing the change of variables $x=x_1$, $y=\dfrac{x_1}{\epsilon}$, $z=\dfrac{x_2}{\epsilon}$,  we obtain that
\begin{equation} \label{RC}
\begin{gathered}
\partial_{x_1}=\partial_x+\frac{1}{\epsilon}\partial_y,\qquad  \partial_{x_2}=\frac{1}{\epsilon}\partial_z \\
\partial_{x_1 x_{1}}=\partial_{xx}+\frac{2}{\epsilon}\partial_{xy}+\frac{1}{\epsilon^2}\partial_{yy},\qquad 
\partial_{x_2 x_{2}}=\frac{1}{\epsilon^2}\partial_{zz}.
\end{gathered}
\end{equation}
Observing from the geometry of $R^\epsilon$ that the unit outward normal vectors to $\partial R^\epsilon$ and to $\partial Y^{*}$, $N^\epsilon= (N^\epsilon_1, N^\epsilon_2)$ and $N=(N_1,N_2)$ respectively are related by
\begin{equation} \label{RN}
N^\epsilon(x_1,x_2) = N \left(\frac{x_1}{\epsilon},\frac{x_2}{\epsilon}\right) \quad \textrm{ a.e. in } \partial R^\epsilon,
\end{equation}
one can see by replacing the expansion \eqref{FE} in the problem \eqref{P}, after equate powers of $\epsilon$, that the differential equations for the functions $w_0$, $w_1$ and $w_2$ are as follows: 
For $w_0$
\begin{equation} \label{eq:w0}
\left\{
\begin{gathered}
-\Delta_{y,z} w_0(x,y,z) = 0 \quad \textrm{ in } Y^* \\
\frac{\partial w_0}{\partial N}(x,y,z)=0 \quad \textrm{ on }  B_1\cup B_2 \\
 w_0(x, \cdot, z) \quad L\text{ - periodic}
\end{gathered}
\right.
\end{equation}
which implies that $w_0$ has to be constant with respect to variables $y, z$, ie, 
$$
w_0(x,y,z)=w_0(x), \quad \forall \, (x,y,z) \in (0, 1) \times Y^*.
$$

Concerning $w_1$, we have
\begin{equation} \label{eq:w1}
\left\{
\begin{gathered}
-\Delta_{y,z} w_1(x,y,z)= 2 \partial_x\partial_y w_0 \quad \textrm{ in } Y^* \\
\frac{\partial w_1}{\partial N}(x,y,g(y))=\frac{g'(y)}{\sqrt{1+ (g'(y))^2}}\frac{dw_0}{dx}(x) \quad \textrm{ on } B_1\\
\frac{\partial w_1}{\partial N}(x,y,z)=0 \quad \textrm{ on } B_2\\
 w_1(x, \cdot, z) \quad L\text{ - periodic}.
\end{gathered}
\right.
\end{equation}
Since $\partial_x\partial_y w_0=0$, denoting by $X(y,z)$ the unique solution (up to a constant) of
\begin{equation} \label{AUXP}
\left\{
\begin{gathered}
-\Delta_{y,z} X(y,z)=0 \quad \textrm{ in } Y^* \\
\frac{\partial X}{\partial N}(y,g(y))= - \frac{g'(y)}{\sqrt{1+ (g'(y))^2}} \quad  \textrm{ on } B_1\\
\frac{\partial X}{\partial N}(y,z)=0 \quad \textrm{ on } B_2\\
 X( \cdot, z) \quad L\text{ - periodic} ,
\end{gathered}
\right.
\end{equation}
we get 
\begin{equation} \label{w1}
w_1(x,y,z) = - X(y,z)\dfrac{dw_0}{dx}(x).
\end{equation}

For $w_2$, the boundary value problem is
$$
\left\{
\begin{gathered}
-\Delta_{y,z} w_2(x,y,z)= f(x)- w_0(x) + 2 \partial_x \partial_y w_1(x,y,z) + \partial_x^2 w_0(x) \quad \textrm{ in } Y^* \\
\frac{\partial w_2}{\partial N}(x,y,g(y))=\frac{g'(y)}{\sqrt{1+ (g'(y))^2}}\frac{\partial w_1}{\partial x}(x,y,g(y)) \quad \textrm{ on } B_1\\
\frac{\partial w_2}{\partial N}(x,y,z)=0 \quad \textrm{ on } B_2\\
 w_2(x, \cdot, z) \quad L\text{ - periodic} ,
\end{gathered}
\right.
$$ 
or, equivalently, upon using the expression for $w_1$ in \eqref{w1}, 
\begin{equation} \label{EMS}
\left\{
\begin{gathered}
-\Delta_{y,z} w_2(x,y,z)= f(x) - w_0(x) + (1 - 2 \partial_yX(y,z))\frac{d^2w_0}{dx^2}(x) \quad \textrm{ in } Y^* \\
\frac{\partial w_2}{\partial N}(x,y,g(y))= - \frac{g'(y)}{\sqrt{1+ (g'(y))^2}}X(y,g(y))\frac{d^2w_0}{dx^2}(x) \quad \textrm{ on } B_1\\
\frac{\partial w_2}{\partial N}(x,y,z)=0 \quad \textrm{ on } B_2\\
 w_2(x, \cdot, z) \quad L\text{ - periodic} .
\end{gathered}
\right.
\end{equation} 

\begin{remark}
These problems are second order partial differential equations in the variables $(y,z) \in Y^*$, with $x \in (0,1)$ playing the role of a parameter.
\end{remark}

Next, by Fredholm alternative we can derive the limiting problem for $w_0$. Indeed, if we take test functions $\phi(x,y,z) = \phi(x)$ in (\ref{EMS}), we get
\begin{eqnarray} \label{EMS1}
0  & = & \int_{Y^*} \phi(x) \, \Big\{ f(x) - w_0(x) + (1 - 2 \partial_yX(y,z))\frac{d^2w_0}{dx^2}(x) \Big\} dy dz \nonumber \\
& & + \int_{\partial Y^*} \phi(x) \frac{\partial w_2}{\partial N} \, dS(y,z).
\end{eqnarray} 

On the other hand, by the Divergence's theorem and (\ref{EMS}),
\begin{eqnarray*}
\int_{Y^*} \phi(x) \frac{d^2w_0}{dx^2}(x) \partial_yX(y,z) dy dz 
& = & \int_{Y^*} \phi(x) \frac{d^2w_0}{dx^2}(x) \nabla_{y,z} X(y,z) \cdot 
\left(
\begin{array}{l}
1 \\
0
\end{array}
\right)
dy dz \\
& = & \int_{\partial Y^*} \phi(x) \, \frac{d^2w_0}{dx^2}(x) \, \partial_y X(y,z) \, N_1 \, dS \\
& = & \int_{\partial Y^*} \phi(x) \, \frac{\partial w_2}{\partial N}(y,z) \, dS.
\end{eqnarray*}
Therefore replacing this expression in (\ref{EMS1}), we reach
\begin{eqnarray}
0  = \int_{Y^*} \phi(x) \, \Big\{ f(x) - w_0(x) + \left( 1 - \partial_yX(y,z) \right) \frac{d^2w_0}{dx^2}(x) \Big\} dy dz, \quad \forall \phi \in C_0^\infty(0,1).
\end{eqnarray}
Since $X$ depends only on $y$ and $z$, we can conclude that
$w_0$ must satisfy
\begin{equation} \label{HOME}
\left\{
\begin{gathered}
-r \frac{d^2w_0}{dx^2}(x) + w_0(x) = f(x), \quad x \in (0,1)\\
w_0'(0)=w_0'(1)=0
\end{gathered}
\right.
\end{equation} 
where 
\begin{equation} \label{HOMC}
r=\frac{1}{|Y^*|}\int_{Y^*} \Big\{ 1 - {\partial_y X}(y,z) \Big\} dydz.
\end{equation}

\begin{remark}
The second order differential equation $(\ref{HOME})$ with the constant coefficient $r$ is called the homogenized equation of problem $(\ref{P})$ 
with the homogenized coefficient $r$.
\end{remark}

Now we use the homogenized equation \eqref{HOME} to describe $w_2$.  
We rewrite problem (\ref{EMS}) as 
\begin{equation}\label{eq:w2carac}
\left\{
\begin{gathered}
-{\rm div}_{y,z} \left( \nabla_{y,z} w_2(x,y,z) - \frac{d^2w_0}{dx^2}(x) 
\left(
\begin{array}{c}
X(y,z) \\
0
\end{array}
\right)
\right)
= \left( 1 - r -\partial_y X(y,z) \right) \frac{d^2w_0}{dx^2}(x)  \\
\frac{\partial w_2}{\partial N}(x,y,g(y))= X(y,g(y)) \, \frac{d^2w_0}{dx^2}(x) \, N_1 \quad \textrm{ on } B_1\\
\frac{\partial w_2}{\partial N}(x,y,z) = 0 \quad \textrm{ on } B_2\\
 w_2(x, \cdot, z) \quad L\text{ - periodic} .
\end{gathered}
\right.
\end{equation}
The linearity of \eqref{eq:w2carac} together with the fact that $\frac{d^2 w_0}{dx^2}$ does not depend on the variables $y$ and $z$ suggest that we look for $w_2(x,y,z)$ of the form
\begin{equation} \label{w2}
w_2(x,y,z) = \theta(y,z) \, \frac{d^2w_0}{dx^2}(x) \quad \text{ for } x \in (0,1) \textrm{ and } (y,z) \in Y^*
\end{equation}
where $\theta$ is the solution of the auxiliary problem
\begin{equation} \label{EMS2}
\left\{
\begin{gathered}
-{\rm div}_{y,z} \left( \nabla_{y,z} \theta(y,z) -  
\left(
\begin{array}{c}
X(y,z) \\
0
\end{array}
\right)
\right)
= 1 - r -\partial_y X(y,z) \quad \textrm{ in } Y^* \\
\left( \nabla_{y,z} \theta(y,z) -  
\left(
\begin{array}{c}
X(y,z) \\
0
\end{array}
\right)
\right) \cdot N = 0 \quad \textrm{ on } B_1 \cup B_2\\
 \theta(\cdot, z) \quad L\text{ - periodic} .
\end{gathered}
\right. 
\end{equation} 
Then, we can use (\ref{w1}) and (\ref{w2}) to introduce the following asymptotic expansion for \eqref{P}:
\begin{equation}\label{eq:assexp}
w^\epsilon(x_1,x_2) = w_0(x_1) - \epsilon \, X\left(\frac{x_1}{\epsilon},\frac{x_2}{\epsilon}\right) \, \frac{dw_0}{dx}(x_1)  + \epsilon^2 \, \theta\left(\frac{x_1}{\epsilon},\frac{x_2}{\epsilon}\right) \, \frac{d^2w_0}{dx^2}(x_{1}) + \ldots
\end{equation}
This expansion will play an essential role in Section \ref{EE} below. 

\begin{remark} \label{FandSC}
According to Bensoussan, Lions and Papanicolaou in \cite{BLP}, the functions $X$ and $\theta$ define the first-order correctors 
\begin{equation}\label{eq:cco1}
\kappa^\epsilon(x_1,x_2) = - \epsilon X\left(\dfrac{x_1}{\epsilon},\dfrac{x_2}{\epsilon}\right) \dfrac{dw_0}{dx_1}(x_1), 
\quad (x_1,x_2) \in R^\epsilon
\end{equation}
and the second-order correctors
\begin{equation}\label{eq:cco2}
\mu^\epsilon(x_1,x_2) = - \epsilon \, X \left(\frac{x_1}{\epsilon},\frac{x_2}{\epsilon} \right) \, \frac{dw_0}{dx}(x_1) + \epsilon^2 \, \theta \left(\frac{x_1}{\epsilon},\frac{x_2}{\epsilon} \right) \, \frac{d^2w_0}{dx^2}(x_1), 
\quad (x_1,x_2) \in R^\epsilon.
\end{equation}
\end{remark}

\begin{remark} \label{XTHETA}
The functions $X$ and $\theta$ are originally defined in the representative cell $Y^*$, but to consider these functions in the thin domain $R^\epsilon$, we use their periodicities at $y$ to extend them to the band
$$
Y = \{ (y,z) \in \R^2 \, | \, y \in \R, \, 0 < z < g(y) \},
$$
and we compose them with the diffeomorphisms 
$$
T^\epsilon: R^\epsilon \mapsto Y \, : \, (x_1,x_2) \to (x_1/\epsilon, x_2/\epsilon).
$$
In the analysis below, with some abuse of notation we will denote these compositions by $X\left(x_1/\epsilon,x_2/\epsilon\right)$ 
and $\theta\left(x_1/\epsilon,x_2/\epsilon\right)$ everywhere for $(x_1,x_2) \in R^\epsilon$. With these considerations we can obtain some estimates on $R^\epsilon$ for $X$ and $\theta$ as well. It is easy to see that

\begin{equation}\label{eq:Xstim}
\begin{split}
\| X \|_{L^2(R^\epsilon)}^2  & =  \int_{R^\epsilon} \left| X({x_1}/\epsilon,{x_2}/\epsilon)  \right|^2 dx_1 dx_2 \\
& \leq  \sum_{k=1}^{1/\epsilon L} \epsilon^2 \int_{Y^*} \left| X(y,z)  \right|^2 dy dz  \\
& \leq {\epsilon}/{L} \| X \|_{L^2(Y^*)}^2.
\end{split}
\end{equation} 
Similarly we can get
\begin{equation}\label{eq:anothestm}
\begin{gathered}
 \| \theta \|^2_{L^2(R^\epsilon)} \leq{{\epsilon}/{L}} \| \theta \|^2_{L^2(Y^*)}, \\
 \| \partial_y X \|^2_{L^2(R^\epsilon)} \leq{{\epsilon}/{L}} \| \partial_y X \|^2_{L^2(Y^*)},  \\
 \| \partial_z X \|^2_{L^2(R^\epsilon)} \leq{{\epsilon}/{L}} \| \partial_z X \|^2_{L^2(Y^*)},  \\
 \| \partial_y \theta \|^2_{L^2(R^\epsilon)} \leq{{\epsilon}/{L}} \| \partial_y \theta \|^2_{L^2(Y^*)}, \\
 \| \partial_z \theta \|^2_{L^2(R^\epsilon)} \leq{{\epsilon}/{L}} \| \partial_z \theta \|^2_{L^2(Y^*)}.
 \end{gathered}
\end{equation}
\end{remark}

\begin{remark}
We can solve the problems $(\ref{eq:w0})$, $(\ref{eq:w1})$, $(\ref{AUXP})$, $(\ref{eq:w2carac})$ and $(\ref{EMS2})$ applying the Lax-Milgram Theorem to the elliptic form
$$
a_{Y^*}(\varphi,\phi) = \int_{Y^*} \nabla_{y,z} \varphi \cdot \nabla_{y,z} \phi \, dy dz, \quad \forall \varphi, \phi \in H^1(Y^*)
$$
on the set $V = V_{Y^*} / \R$ where
$$
V_{Y^*} = \{ \varphi \in H^1(Y^*) \; | \; \varphi \text{ is } L \text{- periodic} \textrm{ with respect to $y$-variable}  \}.
$$
Indeed, the following quantity
$$
\| \varphi \|_V = \left( \int_{Y^*} \left| \nabla \varphi \right|^2 \, dydz \right)^{1/2}
$$
defines a norm on $V$.
\end{remark}

\begin{remark}\label{PCHOME}
Also, we can use the elliptic form $a_{Y^*}$ to show that the homogenized coefficient $r$ is positive. We will perform this here for reader's convenience.
For all $\phi \in V$, we have that the solution $X$ satisfies 
$$
a_{Y^*}(X,\phi) = \int_{{B_{1}}} N_1 \, \phi \, dS.
$$
Recall that $B_1$ is the upper part of the boundary of the basic cell. Consequently, $y_1 - X$ satisfies
\begin{equation} \label{EQB0}
a_{Y^*}(y - X, \phi) = \int_{{B_{1}}} N_1 \phi \, dS 
- \int_{{B_{1}}} N_1 \, \phi \, dS = 0
\end{equation}
for all $\phi \in V$.
Also, we have by $(\ref{HOMC}) $
\begin{eqnarray} \label{EQB1}
r \, |Y^*|  
& = &  \int_{Y^*} \frac{\partial}{\partial y}(y - X(y,z)) \, \frac{\partial y}{\partial y} \, dy dz 
= \int_{Y^*} \nabla(y - X(y,z)) \cdot \nabla y \, dy dz \nonumber \\
& = & a_{Y^*}(y - X, y). 
\end{eqnarray}
Hence, due to relation $(\ref{EQB0})$ with $\phi = - X$ and identity $(\ref{EQB1})$, we get
\begin{equation} \label{AP}
r \, |Y^*| = a_{Y^*}(y - X, y) + a_{Y^*}(y - X, - X) = a_{Y^*}(y - X, y - X) > 0.
\end{equation}
\end{remark}

%------------------------------

\section{First-order Corrector}\label{sec:fir-corr} 

As already noted, the solutions $w^\epsilon$ of \eqref{P} actually do not converge in $H^1$-norms. However, if we ``improve'' $w^\epsilon$ by its first-order corrector, we are able to show
the following result

\begin{theorem}\label{teo:first-order}
Let $w^\epsilon$ be the solution of problem \eqref{P} with $f^\epsilon \in L^2(R^\epsilon)$ satisfying
$$
||| f^\epsilon |||_{L^2(R^\epsilon)} \leq C
$$ 
for some $C > 0$ independent of $\epsilon$. 
Consider the family of functions $\hat f^\epsilon \in L^2(0,1)$ defined by
\begin{equation}\label{def-hat-f}
\hat f^\epsilon (x_1) = \epsilon^{-1} \int_0^{\epsilon g(x_1/\epsilon)} f^\epsilon(x_1,x_2)dx_2 .
\end{equation}
If $\hat f^\epsilon \rightharpoonup \hat f $\,  w-$L^2(0,1)$, then
\begin{equation}\label{eq:conv-f-corr}
\lim_{\epsilon \to 0} ||| w^\epsilon  - w_0 - \kappa^\epsilon |||_{H^1(R^\epsilon)} = 0,
\end{equation} 
where $\kappa^\epsilon$ is the first-order corrector of $w^\epsilon$ defined in Remark $\ref{FandSC}$ and 
$w_0 \in H^2(0,1) \cap C^1(0,1)$ is the  unique solution of the homogenized equation $(\ref{HOME}) $
with 
\begin{equation} \label{F0}
f_0=\dfrac{1}{\hat g} \; \hat f.
\end{equation}
\end{theorem}

\begin{proof} 
By variational formulation of \eqref{P}, we have 
$$
a_\epsilon(\varphi, w^\epsilon) = (\varphi,f^\epsilon)_\epsilon,  \quad \forall \, \varphi \in H^1(R^\epsilon). 
$$
Thus, observing that $w_0 + \kappa^\epsilon \in H^1(R^\epsilon)$\footnote{\text{ Here } $w_0$ \text{ is considered as a function of } $x_1$ \text{ and } $x_2$, \text{ simply with some abuse writing} $w_0(x_1,x_2)=w_0(x_1)$.}, we obtain by symmetry of $a_\epsilon$
\begin{equation}\label{eq:nor-firs-corr}
\begin{split}
|||w^\epsilon - w_0 - \kappa^\epsilon|||^2_{H^1(R^\epsilon)} & = 
a_\epsilon(w^\epsilon - w_0 - \kappa^\epsilon,w^\epsilon - w_0 - \kappa^\epsilon) \\
& =  a_\epsilon(w^\epsilon,w^\epsilon - w_0 - \kappa^\epsilon) - a_\epsilon(w_0 + \kappa^\epsilon,w^\epsilon) + a_\epsilon(w_0 + \kappa^\epsilon,w_0 + \kappa^\epsilon)\\
& = (w^\epsilon-2(w_0 + \kappa^\epsilon) , f^\epsilon)_\epsilon + a_\epsilon(w_0 + \kappa^\epsilon,w_0 + \kappa^\epsilon) .
\end{split}
\end{equation}

Using the change of variables $(x,y) \to (x,y/\epsilon)$ on \cite[Theorem 4.3]{ACPS}, it is easy to see that 
\begin{equation*}
\epsilon^{-1} \| w^\epsilon - w_0 \|_{L^2(R^\epsilon)} \to 0 \textrm{ as } \epsilon \to 0.
\end{equation*}
Consequently, $||| w^\epsilon - w_0 |||_{L^2(R^\epsilon)} \stackrel{\epsilon \to 0} \longrightarrow 0$. Therefore
\begin{equation} \label{*}
(w^\epsilon-w_0,f^\epsilon)_\epsilon \leqslant  |||w^\epsilon-w_0|||_{L^2(R^\epsilon)} |||f^\epsilon |||_{L^2(R^\epsilon)}  \to 0, \quad \text{as } \epsilon \to 0.
\end{equation}
By \eqref{eq:Xstim}, we also obtain
\begin{equation} \label{**}
(\kappa^\epsilon,f^\epsilon)_\epsilon 
\leqslant \epsilon^{-1} \|\kappa^\epsilon \|_{L^2(R^\epsilon)} \|f^\epsilon \|_{L^2(R^\epsilon)} 
\leqslant  \frac{\epsilon\, C}{L^{1/2}}  \| X \|_{L^2(Y^*)} \left\| \dfrac{dw_0}{dx_1} \right\|_{L^\infty(0,1)} \to 0, \quad \text{as } \epsilon \to 0 .
\end{equation}
Now, since $\hat f^\epsilon \rightharpoonup \hat f $\,  w-$L^2(0,1)$ we have
\begin{eqnarray} \label{***}
(w_0,f^\epsilon)_\epsilon & = & \epsilon^{-1} \int_0^1 w_0 (x_1) \int_0^{\epsilon g(x_1/\epsilon) } f^\epsilon (x_1,x_2) \, dx_2dx_1 \nonumber \\
& = & \int_0^1 w_0 (x_1) \hat f^\epsilon(x_1)\, dx_1 \nonumber \\
& \to & \hat{g} \int_0^1 w_0 (x_1) f_0 (x_1) \, dx_1, \quad \text{ as } \epsilon \to 0 .
\end{eqnarray}
Therefore, we get from (\ref{*}), (\ref{**}), (\ref{***}) and (\ref{IP0}) that
\begin{equation} \label{eq:TCH}
(w^\epsilon-2(w_0 + \kappa^\epsilon) , f^\epsilon)_\epsilon \stackrel{\epsilon \to 0} \longrightarrow ( w_0, f_0 )_0.
\end{equation}

Next we show that 
\begin{equation} \label{eqeq}
a_\epsilon(w_0 + \kappa^\epsilon,w_0 + \kappa^\epsilon) \to  a_0(w_0 ,w_0 ) \textrm{ as } \epsilon \to 0.
\end{equation} 
First we compute the limit of
\begin{equation} \label{eq:ARGUE}
\begin{split}
 a_\epsilon(w_0 + \kappa^\epsilon,w_0) & = \epsilon^{-1} \int_{R^\epsilon}\big\{ \nabla(w_0 + \kappa^\epsilon) \cdot \nabla w_0 + (w_0 + \kappa^\epsilon) w_0 \big\} dx_1 dx_2\\
 & = \epsilon^{-1} \int_{R^\epsilon} \left\{ \dfrac{dw_0}{dx_1} - {\partial_y X}\left(\frac{x_1}{\epsilon},\frac{x_2}{\epsilon}\right)\dfrac{dw_0}{dx_1} - \epsilon X\left(\frac{x_1}{\epsilon},\frac{x_2}{\epsilon}\right) \dfrac{d^2w_0}{dx_1^2} \right\} \dfrac{dw_0}{dx_1}  dx_1 dx_2 \\
 &  +  \epsilon^{-1} \int_{R^\epsilon} (w_0 + \kappa^\epsilon) w_0 \, dx_1 dx_2 \\
 & = \epsilon^{-1} \int_{R^\epsilon} \dfrac{dw_0}{dx_1}^2 \left\{ 1 - {\partial_y X} \left(\frac{x_1}{\epsilon},\frac{x_2}{\epsilon}\right)\right\} \, dx_1 dx_2 +  \epsilon^{-1} \int_{R^\epsilon} |w_0|^2 \, dx_1 dx_2 \\
& - \epsilon^{-1} \int_{R^\epsilon} \left\{ \epsilon X\left(\frac{x_1}{\epsilon},\frac{x_2}{\epsilon}\right) \dfrac{dw_0}{dx_1}  \dfrac{d^2w_0}{dx_1^2}  + \epsilon X\left(\frac{x_1}{\epsilon},\frac{x_2}{\epsilon}\right)w_0 \dfrac{dw_0}{dx_1} \right \} dx_1 dx_2
\end{split}
\end{equation}
as $\epsilon \to 0$.
Since
$$
 \epsilon^{-1} \int_{R^\epsilon} \dfrac{dw_0}{dx_1}^2 \left\{ 1 - {\partial_y X} \left(\frac{x_1}{\epsilon},\frac{x_2}{\epsilon}\right) \right\} \, dx_1 dx_2 =  \int_0^1 \int_0^{g(x_1/\epsilon)} \dfrac{dw_0}{dx_1}^2 \left\{ 1 -  {\partial_y X} \left(\frac{x_1}{\epsilon},z \right) \right\} dz  dx_1,
$$
and 
$$
\Phi(y) = \displaystyle \int_0^{g(y)} \left\{ 1 -  {\partial_y X} \left(y,z\right) \right\} \, dz 
$$ 
is a $L$-periodic function, we obtain that
\begin{eqnarray} \label{eq*}
 \epsilon^{-1} \int_{R^\epsilon} \dfrac{dw_0}{dx_1}^2 \left\{ 1 -  {\partial_y X} \left(\frac{x_1}{\epsilon},\frac{x_2}{\epsilon}\right) \right\} \, dx_1 dx_2 & \to & \int_0^1  \dfrac{dw_0}{dx_1}^2  \frac{1}{L} \int_0^L \int_0^{g(y)} \big\{ 1 - {\partial_y X}(y,z) \big\} \, dz  dy dx_1 \nonumber \\
& = & \hat{g} \int_0^1 r \;  \dfrac{dw_0}{dx_1}^2 \, dx_1 \quad \textrm{ as } \epsilon \to 0.
\end{eqnarray}

Notice that we also have
\begin{equation} \label{eq**}
\epsilon^{-1} \int_{R^\epsilon} | w_0 |^2 \, dx_1 dx_2 \to \hat g \int_0^1 | w_0 |^2 dx_1 \quad \textrm{ as } \epsilon \to 0.
\end{equation}
Since $w_0$ does not depend on $x_2$, it follows from \eqref{eq:Xstim} that
\begin{equation}\label{eq:corrfir}
- \epsilon^{-1} \int_{R^\epsilon} \left\{ \epsilon X\left(\frac{x_1}{\epsilon},\frac{x_2}{\epsilon}\right) \dfrac{dw_0}{dx_1}  \dfrac{d^2w_0}{dx_1^2}  + \epsilon X\left(\frac{x_1}{\epsilon},\frac{x_2}{\epsilon} \right) w_0 \dfrac{dw_0}{dx_1} \right\} dx_1 dx_2 \to 0, \text{ as } \epsilon \to 0.
\end{equation}  
Hence, we have from (\ref{eq*}), (\ref{eq**}), (\ref{eq:corrfir}) and (\ref{IPA0}) that
\begin{equation}
a_\epsilon(w_0 + \kappa^\epsilon,w_0) \to a_0(w_0,w_0), \text{ as } \epsilon \to 0.
\end{equation} 

Finally, arguing as in (\ref{eq:ARGUE}), we can obtain from \eqref{eq:Xstim} and \eqref{eq:anothestm} that
$$
a_\epsilon(w_0 + \kappa^\epsilon,\kappa^\epsilon) \to 0, \text{ as } \epsilon \to 0,
$$
getting the statement (\ref{eqeq}).

Therefore, in accordance with \eqref{eq:nor-firs-corr}, we obtain
$$
 |||w^\epsilon - w_0 - \kappa^\epsilon |||^2_{H^1(R^\epsilon)}  \stackrel{\epsilon \to 0}{\longrightarrow} a_0(w_0,w_0) - (w_0 , f_0)_0 = 0
$$
completing the proof.

\end{proof}
\begin{remark}
If the original non homogeneous term $f^\epsilon$ does not depend on $x_2$ and $\epsilon$, ie. $f^\epsilon(x_1,x_2)=f(x_1)$, then it follows from the above definitions $(\ref{def-hat-f})$ and $(\ref{F0})$ that 
$$
\hat f^\epsilon (x_1) = \epsilon^{-1} \int_0^{\epsilon g(x_1/\epsilon)} f(x_1)dx_2 =  g(x_1/\epsilon) f(x_1) \rightharpoonup \hat g f \quad \text{w}-L^2(0,1) .
$$
Hence, equation \eqref{HOME} is in agreement with the one found via the method of Multiple Scales in Section $\ref{MSM}$. 
\end{remark}
\begin{remark} \label{smooth}
In Remark $\ref{PCHOME}$ we show the positiveness of the constant $r$. Hence, the solution $w_0$ of the homogenized equation actually exists, is unique and satisfies $w_0 \in H^2(0,1)  \cap  C^1(0,1)$. 
\end{remark}

%%%------------------------------------

\section{Second-order corrector} \label{EE}

{
Let $w_0$ be the homogenized solution (\ref{HOME}), $X$ and $\theta$ be the auxiliary solutions given by (\ref{AUXP}) and (\ref{EMS2}) 
on the basic cell $Y^*$, which were conveniently defined in the thin domain $R^\epsilon$ by way off Remark \ref{XTHETA}.

In this section, we use the second-order corrector \eqref{eq:cco2} to present an error estimate when we replace the solutions $w^\epsilon$ of \eqref{P} by the first-order truncation 
\begin{equation} \label{AE1}
\mathcal{W}_1^\epsilon(x_1,x_2) = w_0(x_1) - \epsilon \, X \left(\frac{x_1}{\epsilon},\frac{x_2}{\epsilon} \right) \, \frac{dw_0}{dx}(x_1),
\quad (x_1,x_2) \in R^\epsilon
\end{equation}
with respect to the norm $||| \cdot |||_{H^1(R^\epsilon)}$ introduced in Section \ref{PRE}. 
Thus, let us consider the second-order truncation 
\begin{equation} \label{AE}
\mathcal{W}_2^\epsilon(x_1,x_2) = w_0(x_1) - \epsilon \, X \left(\frac{x_1}{\epsilon},\frac{x_2}{\epsilon} \right) \, \frac{dw_0}{dx}(x_1) + \epsilon^2 \, \theta \left(\frac{x_1}{\epsilon},\frac{x_2}{\epsilon} \right) \, \frac{d^2w_0}{dx^2}(x_1), \quad (x_1,x_2) \in R^\epsilon.
\end{equation}

\begin{theorem}  \label{T2}
Let $R^\epsilon$ be the thin domain defined in $(\ref{TDG})$ and let $w^\epsilon$ be the solution of problem 
$(\ref{P})$ with $f^\epsilon(x_1,x_2)=f(x_1)$, $f \in W^{2,\infty}(0,1)$. 

Then, if $\mathcal{W}_1^\epsilon$ and $\mathcal{W}_2^\epsilon$ are given by $(\ref{AE1})$ and $(\ref{AE})$ respectively, we obtain that
\begin{equation} \label{SOA}
||| w^\epsilon - \mathcal{W}_2^\epsilon |||_{H^1(R^\epsilon)} \leq K_1 \, \sqrt{\epsilon}.
\end{equation}
Consequently, we obtain the following rate for the first-order approximation
\begin{equation} \label{FOA}
||| w^\epsilon - \mathcal{W}_1^\epsilon |||_{H^1(R^\epsilon)} \leq K_2 \, \sqrt{\epsilon},
\end{equation}
where $K_1$ and $K_2$ are positive constants independent of $\epsilon > 0$.
\end{theorem}
}
\begin{proof}
First, we note that \eqref{FOA} is a direct consequence from \eqref{SOA} and \eqref{eq:anothestm}, since that
\begin{eqnarray*}
\epsilon^{-1} \, \Big\| \epsilon^2 \theta \, \frac{d^2w_0}{dx^2} \Big\|^2_{H^1(R^\epsilon)} 
& \leq & \epsilon^3 \Big\| \frac{d^2w_0}{dx^2} \Big\|_{L^\infty(0,1)} \, \| \theta \|^2_{L^2(R^\epsilon)} \\
& & + \, \epsilon \, \Big\| \frac{d^3w_0}{dx^3} \Big\|_{L^\infty(0,1)} 
\left(  \| \partial_y \theta \|^2_{L^2(R^\epsilon)} + \| \partial_z \theta \|^2_{L^2(R^\epsilon)} \right)   \\ 
& \leq & K_2 \, \epsilon^2 
\end{eqnarray*}
for some constant $K_2 >0$ independent of $\epsilon$.

Now we estimate the norm $||| \phi^\epsilon |||^2_{H^1(R^\epsilon)} = a_\epsilon(\phi^\epsilon,\phi^\epsilon)$ of the function $\phi^\epsilon$ given by
$$ 
\phi^\epsilon= w^\epsilon - \mathcal{W}_2^\epsilon.$$ 
In order to do it, we compute $a_\epsilon(\phi^\epsilon, \varphi)$ for arbitrary test functions $\varphi \in
H^1(R^\epsilon)$ and establish an estimate of the form 
$$
|a_\epsilon(\phi^\epsilon,\varphi)| \leq K(\epsilon) ||| \varphi |||_{H^1(R^\epsilon)}.
$$
Since $a_\epsilon$ is an elliptic form, we can take $\varphi = \phi^\epsilon$ in the above inequality 
to obtain the desired estimate.
Using the notation from Section \ref{MSM}, we get by (\ref{RC}) that 
\begin{eqnarray*}
\sum_{i=1}^2 \frac{\partial^2 \phi^\epsilon}{{\partial x_i}^2} & = & 
\frac{1}{\epsilon} \, \Delta_{y,z} \left( X \, \frac{dw_0}{dx} \right) \\
& & + \sum_{i=1}^2 \frac{\partial^2 w^\epsilon}{{\partial x_i}^2} - \frac{d^{2}w_0}{dx^{2}}  
+ 2 \, \partial_{xy} \left( X \, \frac{dw_0}{dx} \right) - \Delta_{y,z} \left( \theta \, \frac{d^2w_0}{dx^2} \right) \\
& & + \epsilon \left[ \partial_{xx} \left( X \, \frac{dw_0}{dx} \right) - 2 \, \partial_{xy} \left( \theta \, \frac{d^2w_0}{dx^2} \right) \right] \\
& & - \epsilon^2 \, \partial_{xx} \left( \theta \, \frac{dw_0^2}{dx^2} \right).
\end{eqnarray*}
Hence, due to (\ref{AUXP}) and (\ref{EMS2}), we have
\begin{eqnarray*}
\sum_{i=1}^2 \frac{\partial^2 \phi^\epsilon}{{\partial x_i}^2} & = & 
\sum_{i=1}^2 \frac{\partial^2 w^\epsilon}{{\partial x_i}^2} - \frac{d^{2}w_0}{dx^{2}} 
+ 2 \, \partial_{y} X \, \frac{dw_0^2}{dx^2} + \left(  1 - r - 2 \, \partial_y X \right) \frac{d^2w_0}{dx^2} \\
& & + \epsilon \left[ \frac{d^3w_0}{dx^3} \left( X - 2 \, \partial_{y} \theta \right) \right]
- \epsilon^2 \, \theta \, \frac{d^4w_0}{dx^4}.
\end{eqnarray*}
Consequently, it follows from (\ref{P}) and (\ref{HOME}) (after some calculations) that
\begin{eqnarray*}
- \sum_{i=1}^2 \frac{\partial^2 \phi^\epsilon}{{\partial x_i}^2} + \phi^\epsilon & = &
f + r \, \frac{d^2w_0}{dx^2} - w_0 
- \epsilon \left[ \frac{d^3w_0}{dx^3} \left( X - 2 \partial_y \theta \right) - X \, \frac{dw_0}{dx} \right] \\
& & - \epsilon^2 \left[ \theta \left( \frac{d^2w_0}{dx^2} - \frac{d^4w_0}{dx^4} \right) \right].
\end{eqnarray*}
That is, 
\begin{eqnarray*}
- \sum_{i=1}^2 \frac{\partial^2 \phi^\epsilon}{{\partial x_i}^2} + \phi^\epsilon & = &
- \epsilon \left[ \frac{d^3w_0}{dx^3} \left( X - 2 \partial_y \theta \right) - X \, \frac{dw_0}{dx} \right] 
- \epsilon^2 \left[ \theta \left( \frac{d^2w_0}{dx^2} - \frac{d^4w_0}{dx^4} \right) \right] 
\end{eqnarray*}
in $R^\epsilon$ for all $\epsilon > 0$.

On the boundary $\partial R^\epsilon$, we have by the identity (\ref{RN}) 
and boundary conditions from (\ref{AUXP}) and (\ref{EMS2}) that
\begin{eqnarray*}
\frac{\partial \phi^\epsilon}{\partial N^\epsilon} & = & \nabla_{x_1,x_2} \phi^\epsilon \cdot N^\epsilon \\
& = & \frac{\partial w^\epsilon}{\partial N^\epsilon} - \nabla_{x_1,x_2} \left( w_0 - \epsilon \, X \, \frac{dw_0}{dx}
+ \epsilon^2 \, \theta \, \frac{dw_0^2}{dx^2}  \right) \cdot N^\epsilon \\
& = & - \left(  \partial_x + \frac{1}{\epsilon} \partial_y \right) \left( w_0 - \epsilon \, X \, \frac{dw_0}{dx}
+ \epsilon^2 \, \theta \, \frac{dw_0^2}{dx^2}  \right) \, N_1 \\
& &  - \frac{1}{\epsilon} \partial_z \left(  w_0 - \epsilon \, X \, \frac{dw_0}{dx}
 + \epsilon^2 \, \theta \, \frac{dw_0^2}{dx^2}  \right) \, N_2 \\
& = & \frac{dw_0}{dx} \left( \frac{\partial X}{\partial N} - N_1  \right) 
+ \epsilon \, \frac{dw_0^2}{dx^2} \left( X \, N_1 - \frac{\partial \theta}{\partial N} \right)
- \epsilon^2 \, \theta \, \frac{dw_0^3}{dx^3} \, N_1 \\
& = & - \epsilon^2 \, \theta \, \frac{dw_0^3}{dx^3} \, N_1.
\end{eqnarray*}
Thus, the function $\phi^\epsilon$ satisfies the following boundary value problem
\begin{equation} \label{PPHI}
\left\{
\begin{gathered}
- \Delta \phi^\epsilon + \phi^\epsilon 
= \epsilon \, F^\epsilon \quad \textrm{ in } R^\epsilon \\
\frac{\partial \phi^\epsilon}{\partial N^\epsilon} = \epsilon^2 \, H^\epsilon \, N^\epsilon_1 
\quad \textrm{ on } \partial R^\epsilon
\end{gathered}
\right.
\end{equation}
where 
\begin{equation} \label{EF}
\begin{gathered}
F^\epsilon(x_1,x_2) = - \left[ \frac{d^3w_0}{dx^3}(x_1) \Big( X \left( \frac{x_1}{\epsilon},\frac{x_2}{\epsilon} \right) 
- 2 \partial_y \theta \left(\frac{x_1}{\epsilon},\frac{x_2}{\epsilon} \right) \Big) - X \left(\frac{x_1}{\epsilon},\frac{x_2}{\epsilon} \right) \, 
\frac{dw_0}{dx}(x_1) \right] \\
\quad \quad - \epsilon \left[ \theta \left(\frac{x_1}{\epsilon},\frac{x_2}{\epsilon}\right) \left( \frac{d^2w_0}{dx^2}(x_1) 
- \frac{d^4w_0}{dx^4}(x_1) \right) \right] \quad \text{for a.e. } (x_1,x_2) \in R^\epsilon 
\end{gathered}
\end{equation}
and
\begin{equation} \label{EH}
\begin{gathered}
H^\epsilon(x_1,x_2) = - \theta \left(\frac{x_1}{\epsilon},\frac{x_2}{\epsilon} \right) \, \frac{dw_0^3}{dx^3}(x_1) 
\quad  \text{for a.e. } (x_1,x_2) \in \partial R^\epsilon.
\end{gathered}
\end{equation}

We consider now the variational formulation of problem (\ref{PPHI}): 
find $\phi^\epsilon \in H^1(R^\epsilon)$ such that
\begin{equation} \label{FV}
a_\epsilon(\phi^\epsilon,\varphi) = \int_{R^\epsilon} F^\epsilon \, \varphi \, dx 
+ \epsilon \int_{\partial R^\epsilon} H^\epsilon \, N^\epsilon_1 \, \varphi \, dS.
\end{equation}
Observe that the function $\phi^\epsilon$ must satisfy an uniform a priori estimate on $\epsilon$.
Indeed, if we take $\varphi = \phi^\epsilon$ in the expression (\ref{FV}), we obtain
\begin{eqnarray} \label{EP}
||| \phi^\epsilon |||_{H^1(R^\epsilon)}^2 & = & |a_\epsilon(\phi^\epsilon,\phi^\epsilon)| \\
& \leq  & \| \phi^\epsilon \|_{L^2(R^\epsilon)} \| F^\epsilon \|_{L^2(R^\epsilon)}
+ \epsilon \, \| \phi^\epsilon \|_{L^2(\partial R^\epsilon)} \| H^\epsilon N_1^\epsilon \|_{L^2(\partial R^\epsilon)}.
\nonumber
\end{eqnarray}

We need to get sharp inequalities on $F^\epsilon$ and $H^\epsilon$ to estimate $a_\epsilon(\phi^\epsilon,\phi^\epsilon)$.
It is clear from their definitions that these estimates will be consequence of those ones for $w_0$, $X$ and $\theta$.
Since $f$ is a smooth function, we have by classical regularity results given in \cite{ADN} that the solution $w_0$ of the 
homogenized problem is smooth enough to guarantee that its derivatives up to the fourth order are in $L^\infty(0,1)$. 
Note that similar statements are also true for $X$ and $\theta \in H^1(Y^*)$.

Due to the periodicity of $X$, we have by (\ref{eq:Xstim}) that
$$
\begin{gathered}
|| X ||_{L^2(R^\epsilon)}^2 \leq \frac{\epsilon}{L} || X ||_{L^2(Y^*)}^2, \;
|| \theta ||_{L^2(R^\epsilon)} \leq \sqrt{\frac{\epsilon}{L}} || \theta ||_{L^2(Y^*)} \; \textrm{ and } \; 
|| \partial_y \theta ||_{L^2(R^\epsilon)} \leq \sqrt{\frac{\epsilon}{L}} || \theta ||_{L^2(Y^*)}.
\end{gathered}
$$
Consequently, it is clear from (\ref{EF}) that there exists $K_0$ independent of $\epsilon$ such that
\begin{equation} \label{DF}
\| F^\epsilon \|_{L^2(R^\epsilon)} \leq K_0 \, \sqrt{\epsilon}.
\end{equation}
Let us observe that $K_0$ depends on the period $L$ of the norms of $ X, \, \theta \textrm{ and } \partial_y \theta $ in $L^{2}(Y^*)$, as well of the norms of $\frac{dw_0}{dx}, \, \frac{d^2w_0}{dx^2}, \, \frac{d^3w_0}{dx^3} \textrm{ and } \frac{d^4w_0}{dx^4} $ in $L^{\infty}(0,1)$.

Now, let us denote the {\it oscillatory} part of $\partial R^\epsilon$ by 
$\partial_{o} R^\epsilon = \{(x_1,\epsilon g(x_1/\epsilon)):0<x_1<1\},$
the {\it fixed} part by 
$\partial_{f} R^\epsilon = \{(x_1,0)):0<x_1<1\}$ 
and the {\it lateral} part of $\partial R^\epsilon$ as 
$\partial_{l} R^\epsilon = \{(0,x_2)):0<x_2<\epsilon g(0)\}\cup \{(1,x_2)):0<x_2<\epsilon g(1/\epsilon)\}.$
From definition (\ref{EH}) we have
\begin{eqnarray*}
\| H^\epsilon \, N_1^\epsilon \|_{L^2(\partial R^\epsilon)}^2 & = &
\int_{\partial R^\epsilon} \left| \theta\left(\frac{x_1}{\epsilon},\frac{x_2}{\epsilon}\right) 
\, \frac{dw_0^3}{dx^3}(x_1) \, N_1^\epsilon(x_1,x_2) \right|^2 dS \\
& \leq & \left\| \frac{dw_0^3}{dx^3} \right\|_{L^\infty(0,1)} \int_{\partial R^\epsilon} 
\left| \theta\left(\frac{x_1}{\epsilon},\frac{x_2}{\epsilon}\right) \right|^2 dS \\
& \leq & \left\| \frac{dw_0^3}{dx^3} \right\|_{L^\infty(0,1)} 
\left( \int_{\partial_{o}R^\epsilon}  \left| \theta\left(\frac{x_1}{\epsilon},\frac{x_2}{\epsilon}\right) \right|^2 dS 
+ \int_{\partial_{f}R^\epsilon}  \left| \theta\left(\frac{x_1}{\epsilon},\frac{x_2}{\epsilon}\right) \right|^2 dS \right) \\
& \leq & K_1 \left( \sum_{k=1}^{1/\epsilon L} \, \epsilon \, \int_0^L  \left| \theta(y,g(y)) \right|^2 \, dy 
+ \sum_{k=1}^{1/\epsilon L} \, \epsilon \, \int_0^L  \left| \theta(y,0) \right|^2 dy \right) \\
& \leq & \frac{K_1}{L} \| \theta \|_{L^2(\partial Y^*)}^2
\end{eqnarray*}
where 
$
K_1=\| (1+g')^{\frac{1}{2}}\|_{L^{\infty}(0,L)} \left\|{dw_0^3}/{dx^3} \right\|_{L^\infty(0,1)} 
$ 
is independent of $\epsilon$. 
Note that we have used the periodicity of $\theta$ to get
$
\int_{\partial_{l}R^\epsilon}  \left| \theta(\frac{x_1}{\epsilon},\frac{x_2}{\epsilon}) \right|^2 dS = 0.
$
Consequently there exists $K_2 > 0$ independent of $\epsilon$ such that
\begin{equation} \label{DH}
\| H^\epsilon \, N_1^\epsilon \|_{L^2(\partial R^\epsilon)}^2 \leq \tilde K_2. 
\end{equation}

Now we have all the ingredients to estimate $a_\epsilon(\phi^\epsilon,\phi^\epsilon)$.
Due to (\ref{DF}) and (\ref{DH}) we get from (\ref{EP}) that
\begin{equation} \label{EI}
||| \phi^\epsilon |||_{H^1(R^\epsilon)}^2 \leq  \epsilon^{1/2} \, K_0 \, \| \phi^\epsilon \|_{L^2(R^\epsilon)} 
+ \epsilon \, \tilde K_0 \, \| \phi^\epsilon \|_{L^2(\partial R^\epsilon)}.
\end{equation}
Hence, the desired result follows from the following fact:
If $\varphi \in H^1(R^\epsilon)$, then there exists a constant $C$ independent of $\epsilon$ such that
\begin{equation} \label{EI1}
\| \varphi \|_{L^2(\partial R^\epsilon)} \leq C \epsilon^{-1/2} \, \| \varphi \|_{H^1(R^\epsilon)}.
\end{equation}
Indeed, if we combine (\ref{EI}) and (\ref{EI1}), we obtain $K_1>0$ independent of $\epsilon$ 
such that 
$$
||| \phi^\epsilon |||_{H^1(R^\epsilon)} \leq  K_1 \, \epsilon^{1/2}.
$$
The proof of (\ref{EI1}) can be found in \cite{CP,L}.
We recall it here for the reader's convenience. From smoothness of $\partial Y^*$ we can define a smooth extension, $M = (M_1,M_2) \in C^{1}(\overline{Y^*})$, of the unitary normal vector field $N$ on $Y^*$, such that  $M(y,z) = N(y,z) \textrm{ a.e. } \partial Y^*$, and with support of $M$ in a some neighborhood of $\partial Y^*$.
Hence, for all $\varphi \in H^1(R^\epsilon)$ it follows that
\begin{eqnarray*}
\| \varphi \|_{L^2(\partial R^\epsilon)}^2 
& = & \int_{\partial R^\epsilon} \varphi^2 \, M \left( \frac{x_1}{\epsilon},\frac{x_2}{\epsilon} \right) \cdot N^\epsilon \, dS \\
& = & \int_{R^\epsilon} \nabla \left(\varphi^2\right) \cdot M \left( \frac{x_1}{\epsilon},\frac{x_2}{\epsilon} \right) dx_1 dx_2 +
\int_{R^\epsilon} \varphi^2 \, {\rm div} M \left( \frac{x_1}{\epsilon},\frac{x_2}{\epsilon} \right) dx_1 dx_2 \\
& = & 2 \int_{R^\epsilon} \varphi \nabla \varphi \cdot M \left( \frac{x_1}{\epsilon},\frac{x_2}{\epsilon} \right)  dx_1 dx_2  + \epsilon^{-1} \int_{R^\epsilon} \varphi^2  \left\{ \sum_{i=1}^2 \partial_{y_i} M_i \left( \frac{x_1}{\epsilon},\frac{x_2}{\epsilon} \right) \right\} dx_1 dx_2 \\
& \leq & C_1 \, \|\varphi\|_{L^2(R^\epsilon)} \, \|\nabla \varphi\|_{L^2(R^\epsilon)} + C_2{\epsilon^{-1}} \, \|\varphi\|^{2}_{L^2(R^\epsilon)}  \\
& \leq & \max\{C_1,C_2\} \, {\epsilon^{-1}}\|\varphi\|_{H^1(R^\epsilon)}^{2}.
\end{eqnarray*}
\end{proof}

\section{Final conclusion}

In this work, we give a precise rate of the convergence for solutions of an elliptic problem posed in a family of rough domains with a singular collapsing structure. In our analysis we use the corrector approach grounded in a formal asymptotic expansion of solutions, widely used in homogenization theory, to obtain a rigorous strong convergence result in Theorem \ref{teo:first-order}. The second-order corrector is used to obtain an error estimate for the convergence result in Theorem \ref{T2}. We observe that the rate of our convergence result is sharp in the sense that we do not have introduced any boundary layer terms in the correctors, see e.g. \cite{GP}.

\vspace{1 cm}

\par\noindent {\bf Acknowledgments.} Part of this work was done while the authors were visiting the Departamento de Matem\'atica of the Universidade de S\~ao Paulo, SP - Brazil. We kindly express our gratitude to the Department.  The authors thanks to Jos\'e M. Arrieta and Alexandre N. Carvalho for their remarks and suggestions. We also would like to thank the anonymous referee whose comments have considerably improved the writing of the paper.


\begin{thebibliography}{99}

\bibitem{ADN} S. Agmon, A. Douglis, L. Nirenberg; 
\emph{Estimates near the boundary for solutions of elliptic partial diffe\-ren\-tial equations satisfying general boundary value conditions $I$}; Comm. Pure Appl. Math. 12 (1959) 623-727.

\bibitem{ABMG} Y. Amirat, O. Bodart, U. de Maio, A. Gaudiello; \emph{Asymptotic Approximation 
of the solution of the Laplace equation in a domain with highly oscillating boundary}; \emph{SIAM J. Math. Anal.} 
35, (2004) 1598-1616.

\bibitem{A} J. M. Arrieta;
\emph{Spectral properties of Schr\"{o}dinger operators under  perturbations of the domain}; Ph.D. Thesis, Georgia Inst. of Tech. (1991).

\bibitem{ACPS} J. M. Arrieta, A. N. Carvalho, M. C. Pereira and R. P. Silva; 
\emph{Nonlinear parabolic problems in thin domains with a highly oscillatory boundary}; Nonlinear Analysis: Theory Methods and Appl., (74) 15 (2011) 5111-5132.

\bibitem{AP} J. M. Arrieta and M. C. Pereira; 
\emph{Elliptic problems in thin domains with highly oscillating boundaries}; Bol. Soc. Esp. Mat. Apl. 51 (2010) 17-25.

\bibitem{AP2} J. M. Arrieta and M. C. Pereira; 
\emph{Homogenization in a thin domain with an oscillatory boundary}; J. Math. Pures et Appl. 96 (2011) 29-57.

\bibitem{BLP} A. Bensoussan, J. L. Lions and G. Papanicolaou; 
\emph{Asymptotic Analysis for Periodic Structures}; North-Holland (1978).

\bibitem{BCh} R. Brizzi, J.P. Chalot; \emph{Boundary homogenization and Neumann boundary problem}; Ricerce di Matematica XLVI, 2 (1997) 341-387.

%\bibitem{CR} I. D. Chueshov and A. M. Rekalo; 
%\emph{Global attractor of a contact parabolic problem in a thin two-layer domain}; Sbornik: Mathematics 195 (1) (2001) 97-119.

\bibitem{CDG} D. Cioranescu, A. Damlamian and G. Griso; 
\emph{The periodic unfolding method in homogenization}; SIAM J. Math. Anal. 40 no. 4 (2008) 1585-1620.

\bibitem{CD} D. Cioranescu and P. Donato; 
\emph{An Introduction to Homogenization}; Oxford lecture series in mathematics and its applications (1999).

\bibitem{CP-79} D. Cioranescu and J. Saint J. Paulin; \emph{Homogenization in open sets with holes}; J. Math Anal. Appl. 71 (1979), 590-607.

\bibitem{CP} D. Cioranescu and J. Saint J. Paulin; 
\emph{Homogenization of Reticulated Structures}; Springer Verlag (1980).

\bibitem{DP}  A. Damlamian, K. Pettersson; \emph{Homogenization of oscillating boundaries}; 
 Discrete and Continuous Dynamical Systems 23, (2009), 197-219.

\bibitem{Elsken:2005p40} T. Elsken;
\emph{Continuity of attractors for net-shaped thin domain}; Topol. Meth. Nonlinear Analysis, 26 (2005) 315-354.

\bibitem{HR} J. K. Hale and G. Raugel; 
\emph{Reaction-diffusion equation on thin domains}; J. Math. Pures et Appl. (9) 71 no. 1 (1992) 33-95.

\bibitem{L} J. -L. Lions; 
\emph{Asymptotic expansions in perforated media with a periodic structure}; Rocky Mountain J. Math. 10 (1) (1998) 125-140.

\bibitem{Madureira1} D. N. Arnold and A. L. Madureira; \emph{Asymptotic Estimates of Hierarchical}; Modeling Mathematical Modeling and Methods in Applied Sciences, vol. 13, No. 9, (2003) 1325-1350.

%\bibitem{Madureira2} A. L. Madureira; \emph{Hierarchical Modeling Based on Mixed Principles: Asymptotic Error Estimates}; Mathematical Modeling and Methods in Applied Sciences, vol. 15, no.7, (2005) 985-1008. 

\bibitem{Madureira3} A. L. Madureira and F. Valentin; \emph{Asymptotics of the Poisson Problem in domains with curved rough boundaries}; SIAM Journal on Mathematical Analysis 38, No. 5, (2007) 1450-1473.

\bibitem{TAM} T. A. Mel'nyk;
\emph{Homogenization of the Poisson equation in a thick periodic junction}; 
Z. Anal. Anwendungen 18, (4), (1999) 953-975. 


\bibitem{N-K:97} J. Nevard and J. B. Keller;
\emph{Homogenization of rough boundaries and interfaces}; SIAM J. Appl. Math., 57 no. 6 (1997) 1660-1686


\bibitem{GP} G. Panasenko;
\emph{Multi-scale modelling for structures and composites}; Springer, Dordrecht, 2005.

\bibitem{PR01} M. Prizzi and K. P. Rybakowski;
\emph{The effect of domain squeezing upon the dynamics of reaction-diffusion equations}; Journal of Diff. Equations, 173 (2) (2001) 271-320.

\bibitem{Prizzi:2002p441} M. Prizzi and M. Rinaldi and K. P. Rybakowski;
\emph{Curved thin domains and parabolic equations}; Studia mathematica, 151 (2) (2002) 109-140.

\bibitem{R} G. Raugel; 
\emph{Dynamics of partial differential equations on thin domains}; Lecture Notes in Math., vol 1609, Springer Verlag (1995).

\bibitem{SP} E. S\'anchez-Palencia;  
\emph{Non-Homogeneous Media and Vibration Theory}; Lecture Notes in Phys. 127, Springer Verlag (1980).

\bibitem{Ricardo} R. P. Silva; 
\emph{Semicontinuidade inferior de atratores para problemas parab\'olicos em dom\'inios finos}; Phd Thesis, ICMC - USP, (2007).

\bibitem{Tt} L. Tartar; 
\emph{ The General Theory of Homogenization. A personalized introduction}; Lecture Notes of the Un. Mat. Ital, 7, Springer-Verlag, Berlin (2009).

\end{thebibliography}
\end{document}